\documentclass{amsart}

\newtheorem{theorem}{Theorem}[section]
\newtheorem{lemma}[theorem]{Lemma}
\newtheorem{proposition}{Proposition}[section]
\newtheorem{Corollary}{Corollary}[section]
\theoremstyle{definition}
\newtheorem{definition}[theorem]{Definition}
\newtheorem{example}[theorem]{Example}

\theoremstyle{proposition}
\theoremstyle{Corollary}
\theoremstyle{remark}
\newtheorem{remark}[theorem]{Remark}

\numberwithin{equation}{section}

\begin{document}

\title[   Metallic semi-Riemannian manifold]{Invariant and screen semi-invariant lightlike submanifolds of a metallic semi-Riemannian manifold with a quarter symmetric non-metric connection}

\author{Jasleen Kaur}
\address{Department of Mathematics, Punjabi University, Patiala}
\email{jasleen$\_$math@pbi.ac.in}

\author{Rajinder Kaur}
\address{Research Scholar, Department of Mathematics, Punjabi University, Patiala}
\email{rajinderjasar@gmail.com}

\subjclass[2020]{53C12, 53C15, 53C40, 53C50}

\begin{abstract}
This research work introduces the structure of invariant  and screen semi-invariant lightlike submanifolds of a metallic semi-Riemannian manifold with a quarter symmetric non-metric connection, elaborated with examples. It delves into the characterization of integrability and parallelism of distributions inherent in the structure of these submanifolds. Additionally, it presents findings pertaining to totally geodesic foliations  for invariant and screen semi-invariant submanifolds.

\end{abstract} 

\maketitle

\section{INTRODUCTION}\label{sec1}

A significant class of manifolds, namely  metallic Riemannian manifolds which have emerged from the metallic numbers of the metallic means family, is an effective domain of study in the realm of differential geometry. The metallic numbers  exhibit important mathematical properties that constitute a bridge between mathematics and design.	The \textquotedblleft{Metallic Means Family}\textquotedblright introduced by Vera W.de Spinadel in \cite{V1} encompasses a range of means, including the golden mean, silver mean, bronze mean, copper mean, and others, which are defined in terms of values of metallic numbers. These means have been extensively examined for their mathematical properties and their applications across various fields of research, as seen in  \cite{V1},\cite{V2},\cite{V3},\cite{V4},\cite{V5},\cite{V6}.The polynomial structures on manifold were introduced by \cite{P1},\cite{P2} and a specific class known as metallic structure was introduced on a Riemannian manifold by \cite{M1}.\\

The lightlike geometry of submanifolds initiated by \cite{B1} being extremely relevant in different branches of mathematics equipped with degenerate metric has led to the development of many interesting and remarkable results in the fields where the non degenerate metric is not applicable. Since the tools used to investigate the geometry of submanifolds in a Riemannian manifold are not favourable in semi-Riemannian cases, so the lightlike (degenerate) geometry  plays a pivotal role in the study of such structures. Within this framework, the invariant submanifolds for golden semi-Riemannnian manifolds were  investigated by \cite{GSRM5} and for metallic semi-Riemannian manifolds with metric connection by \cite{M2}. One more interesting class of submanifolds, namely screen semi-invariant lightlike submanifold was established by \cite{IMP4} for the semi-Riemannian product manifold. \cite{GSRM9} initiated the geometry of screen semi-invariant lightlike submanifold for golden semi-Riemannian manifold which were further worked upon by \cite{GRM DOUBLE STAR}.\\

A quarter symmetric linear connection introduced by \cite{Q2} is defined as: A linear connection $\bar{\nabla}$ on a Riemannian manifold $(\bar{M},\bar{g})$ is said to be a quarter symmetric connection if its torsion tensor $\bar{T}$ satisfies \\
\begin{equation*}
	\bar{T}(X,Y)=\pi(Y)\phi(X)-\pi(X)\phi(Y)
\end{equation*}
where $\phi$ is a $(1,1)$-tensor field and $\pi$ is a 1-form associated with a smooth unit vector $\xi$, called the characteristic vector field, by $\pi(X)=\bar{g}(X,\xi)$. If the linear connection $\bar{\nabla}$ is not a metric connection, then $\bar{\nabla}$ is called a quarter symmetric non-metric connection.  Various researchers \cite{Q3},\cite{Q1},\cite{Q5},\cite{Q6} have developed the geometry of submanifolds for semi-Riemannian manifolds equipped with quarter symmetric non-metric connection. In this context, the screen semi-invariant lightlike submanifolds for product semi-Riemannian manifold were introduced and studied extensively by\cite{Q4}.\\ 

The present paper aims to initiate the geometry of invariant and screen semi-invariant lightlike submanifolds of metallic semi-Riemannian manifolds endowed with quarter symmetric non-metric connection. Some results on the metallic structure have been developed. The integrability and parallelism of distributions have been characterized. The totally geodesic foliations  for invariant and screen semi-invariant submanifolds have also been dealt with. Examples elaborating the structure of the invariant and screen semi-invariant lightlike submanifolds have been presented.

\section{PRELIMINARIES}\label{sec2}

Following \cite{B1}, this section deals with some basic concepts of lightlike geometry  to be used throughout the paper. 

Let ($ \bar{M},\bar g$) be an $(m+n)$-dimensional semi-Riemannian manifold with semi-Riemannian metric $\bar{g}$ and of constant index $q$ such that $m,n\geq 1$, $1\leq q\leq m+n-1$.\\
Let $(M,g)$ be a $m$-dimensional lightlike submanifold of $\bar{M}$. In this case, there exists a smooth distribution $RadTM$ on $M$ of rank $r>0$, known as radical distribution on $M$ such that $Rad TM_p = TM_p \cap TM_p^{\perp}, \forall ~p \in M$ where $TM_p$ and  $TM_p^{\perp}$ are degenerate orthogonal spaces but not complementary. Then $M$ is called an $r$-lightlike submanifold of $\bar{M}$. \\

Now, consider $S(TM)$, known as screen distribution, as  a complementary distribution of radical distribution in  $TM$  i.e.,
\[
TM = Rad TM  \perp S(TM)
\]
and  $S(TM^{\perp})$, called screen transversal vector bundle, as a complementary vector subbundle to $Rad(TM)$ in $TM^{\perp}$ i.e.,
\[
TM^{\perp} = RadTM  \perp S(TM^{\perp})
\]
As $S(TM)$ is non degenerate vector subbundle of $T\bar{M}{\mid}_M$, we have
\[
T\bar{M}{\mid}_M = S(TM) \perp S(TM)^{\perp}
\]
where $S(TM)^{\perp}$  is the complementary orthogonal vector subbundle of $S(TM)$ in $T\bar{M}{\mid}_M$.\\

Let $tr(TM)$ and $ltr(TM)$ be complementary vector bundles to $TM$ in 
$T\bar{M}{\mid}_M$ and to $RadTM$ in $S(TM^{\perp})^{\perp}$ . Then we have
\[
tr(TM) = ltr(TM) \perp S(TM^{\perp})
\]
\[
T\bar{M}{\mid}_M = TM \oplus tr(TM) 
\]
\[ = (RadTM \oplus ltr(TM)) \perp S(TM) \perp S(TM^{\perp}).
\]

\begin{theorem}
	\cite{B1} Let $(M,g,S(TM), S(TM^{\perp}))$ be an $r$-lightlike submanifold of a semi-Riemannian manifold $(\bar{M},\bar{g})$. Then there exists a complementary vector bundle $ltr(TM)$ called a lightlike transversal bundle of $Rad(TM)$ in $S(TM^{\perp})^{\perp}$ and basis of $\Gamma(ltr(TM){\mid}_U)$ consisting of smooth sections $\{N_1,\cdots,N_r\}$ $S(TM^{\perp})^{\perp}{\mid}_U$ such that 
	\[
	\bar{g}(N_i,\xi_j) = \delta_{ij} , \quad \bar{g}(N_i,N_j) = 0, \quad i,j=0,1,\cdots , r
	\]
	where  $\{{\xi_1, \cdots , \xi_r}\}$ is a lightlike basis of $\Gamma(RadTM){\mid}_U$.
\end{theorem}

Let $\bar{\nabla}$ be the Levi-Civitia connection on $\bar{M}$. We have, from the above mentioned theory, the Guass and Weingarten formulae as:
\begin{equation*}\label{eq21}
	\bar\nabla_U V = \nabla_U V +h(U,V) \quad \forall ~ U, V\in \Gamma(TM)
\end{equation*}
and
\begin{equation*}
	\bar\nabla_U N = -A_N U + \nabla_U^{t}N \quad \forall ~ U \in \Gamma(TM), N \in \Gamma(tr(TM))
\end{equation*}\\
where $\{\nabla_U V,-A_N U\}$ and $\{h(U,V),\nabla_U^{t}N\}$ belong to $\Gamma(TM)$ and $\Gamma(tr(TM))$ respectively. $\nabla$ and $\nabla^{t}$ are linear connections on $M$ and on the vector bundle $tr(TM)$.\\

Considering the projection morphisms $L$ and $S$ of $tr(TM)$ on $ltr(TM)$ and on $S(TM^{\perp})$, we have\\
\begin{equation}\label{1}
	\bar\nabla_U V = \nabla_U V +h^{l}(U,V)+h^{s}(U,V)\\
\end{equation}
\begin{equation}\label{2}
	\bar\nabla_U N = -A_N U + \nabla_U^{l}N+D^{s}(U,N)\\
\end{equation}
\begin{equation}\label{3}
	\bar\nabla_U W = -A_W U + \nabla_U^{s}W+D^{l}(U,W)\\
\end{equation}\\
where $h^{l}(U,V)=Lh(U,V) ,h^{s}(U,V)=Sh(U,V), \{\nabla_U V,A_N U,A_W U \}\in\Gamma(TM)$,\\$\{\nabla_U^{l}N,D^{l}(U,W)\}\in\Gamma(ltr(TM))$ and $\{\nabla_U^{s}W,D^{s}(U,N)\}\in\Gamma(S(TM^{\perp}))$.Then considering $(\ref{1})-(\ref{3})$ and the fact that $\bar{\nabla}$ is a metric connection, the following holds: \\
\begin{equation*}\label{4}
	\bar{g}(h^{s}(U,V),W)+\bar{g}(V,D^{l}(U,W))=\bar{g}(A_W U,V)
\end{equation*}
\begin{equation*}\label{5}
	\bar{g}(D^{s}(U,N),W)=\bar{g}(A_W U,N).
\end{equation*}\\
Let $J$ be a projection of $TM$ on $S(TM)$.Then we have\\
\begin{equation*}\label{6}
	\nabla_U JV= \nabla_U^{*}JV + h^{*}(U,JV)
\end{equation*}
\begin{equation*}\label{7}
	\nabla_U E=-A_E^{*}U+\nabla_U^{*t}E
\end{equation*}\\
for any $U,V\in\Gamma(TM)$ and $E\in\Gamma(Rad(TM))$, where$\{\nabla_U^{*}JV,A_E^{*}U\}$ and $\{h^{*}(U,JV),\nabla_U^{*t}E\}$ belong to $\Gamma(S(TM))$ and $\Gamma(Rad(TM))$ respectively.\\
Using the above equations, we obtain 
\begin{equation*}\label{8}
	\bar{g}(h^{l}(U,JV),E)=g(A_E^{*}U,JV)
\end{equation*}
\begin{equation*}\label{9}
	\bar{g}(h^{*}(U,JV),N)=g(A_N U,JV)
\end{equation*}
\begin{equation*}\label{10}
	\bar{g}(h^{l}(U,E),E)=0, \quad A_E^{*}E=0.
\end{equation*}\\
In general, $\nabla$ on $M$ is not metric connection. Since $\bar{\nabla}$ is a metric connection, it follows from $(2.3)$ that \\
\begin{equation*}\label{11}
	(\nabla_U g)(V,Z)=\bar{g}(h^{l}(U,V),Z)+\bar{g}(h^{l}(U,Z),V)
\end{equation*}\\
for any $U,V,Z\in\Gamma(TM)$. Here $\nabla^{*}$ is a metric connection on $S(TM)$.\\

\subsection{Metallic semi-Riemannian manifold}\label{subsec1}\

Some  polynomial structures naturally arise as $C^{\infty}$
tensor fields $\bar{J}$ of type $(1, 1)$ which are roots of the algebraic equation
\[Q(\bar{J}) := \bar{J}^{n} + a_{n}\bar{J}^{n-1} +...... + a_{2}\bar{J} + a_{1}I_{\mathfrak{X}(\bar{M})} = 0\]
where $I_{\mathfrak{X}(\bar{M})}$ is the identity map on the Lie algebra of vector fields on $\bar{M}$. In particular, if the structure polynomial is $Q(\bar{J})= \bar{J}^{2}-p\bar{J} -qI_{\mathfrak{X}(\bar{M})}$
, with $p$ and $q$ positive integers, its solution $\bar{J}$ is called a metallic structure. For different values of $p$ and $q$, the $(p, q)$ metallic number introduced by \cite{V3} is the positive root of the quadratic equation $x^{2}-px-q = 0$, namely 
$\sigma_{p,q}=\dfrac{p+\sqrt{p^{2}+4q}}{2}$ and is called the metallic mean.\\
A polynomial structure on a semi-Riemannian manifold $\bar{M}$ is known metallic if it is determined by $\bar{J}$ such that
\begin{equation}\label{a}
	\bar{J}^{2}=p\bar{J}+qI
\end{equation}
If a semi-Riemannian metric $\bar{g}$ satisfies the equation
\begin{equation}\label{b}
	\bar{g}(U,\bar{J}V)=\bar{g}(\bar{J}U,V)	
\end{equation}
which yields
\begin{equation}\label{c}
	\bar{g}(\bar{J}U,\bar{J}V)=p\bar{g}(U,\bar{J}V)+q\bar{g}(U,V)	
\end{equation},
then $\bar{g}$ is called $\bar{J}$-compatible.

\begin{definition}\cite{MSRM6} A semi-Riemannian manifold $(\bar{M},\bar{g})$ equipped with $\bar{J}$ such that the semi-Riemannian metric $\bar{g}$ is $\bar{J}$-compatible is called metallic semi-Riemannian manifold and $(\bar{g},\bar{J})$ is called metallic structure on $\bar{M}$.\\
	
	Let $(M,g,S(TM),S(TM^{\perp}))$ be a lightlike submanifold of a metallic semi-Riemannian manifold $(\bar{M},\bar{g},\bar{J})$.\\
	For each $U$ tangent to $M$, $\bar{J}U$ can be written as follows:
	\begin{equation}\label{t1}
		\bar{J}U=fU+wU= fU+w_{l}U+w_{s}U
	\end{equation}
	where $fU$ and $wU$ are the tangential and the transversal parts of $\bar{J}U$; $w_{l}$ and $w_{s}$ are projections on $ltr(TM)$ and $S(TM^{\perp})$ respectively. In addition, for any $V\in\Gamma(tr(TM))$, $\bar{J}V$ can be written as
	\begin{equation}\label{t2}
		\bar{J}V=BV+CV
	\end{equation}
	where $BV$ and $CV$ are the tangential and the transversal parts of $\bar{J}V$.\\

\end{definition}	

\subsection{Quarter symmetric non-metric connection}\label{subsec2}\

As per \cite{Q1}, for a Levi-Civita connection $\bar{D}$ on the metallic semi-Riemannian manifold $\bar{M}$,
by setting \begin{equation}\label{a1}
	\bar D_{U}{V}=\bar\nabla_{U}{V}+\pi(V)\bar{J}U
\end{equation}
for any $U,V\in\Gamma(T\bar{M})$ ,
we see that $\bar{D}$ is linear connection on $\bar{M}$, where $\pi$ is a 1-form on $\bar{M}$ with $\eta$ as associated vector field such that
\begin{equation*}
	\pi(U)=\bar{g}(U,\eta)
\end{equation*}
Let the torsion tensor of $\bar{D}$ on $\bar{M}$ be denoted by $\bar{T}$.
\begin{equation*}
	\bar{T}^{\bar{D}}(U,V)=\pi(V)\bar{J}U-\pi(U)\bar{J}V
\end{equation*}

\begin{equation*}
	(\bar D_{U} \bar g)(V,Z)=-\pi(V)\bar{g}(\bar{J}U,Z)-\pi(Z)\bar{g}(\bar{J}U,V)
\end{equation*}
and
\begin{equation}\label{a2}
	\bar D_{U} \bar{J}V = \bar{J}\bar D_{U}V-p\pi(V)\bar{J}U-q\pi(V)U+\pi(\bar{J}V)\bar{J}U
\end{equation}
for any $U,V,Z\in\Gamma(TM)$.Thus $\bar{D}$ is a quarter-symmetric non-metric connection on $\bar{M}$.\\

Consider a lightlike submanifold $(M,g,S(TM),S(TM^{\perp}))$ of the metallic semi-Riemannian manifold $(\bar{M},\bar{g})$ with quarter symmetric non-metric connection $\bar{D}$. Then the Gauss and Weingarten formulae with respect to $\bar{D}$ are given by\\
\begin{equation}\label{q1}
	\bar D_{U}V=D_U V+\bar{h^{l}}(U,V)+\bar{h^{s}}(U,V)
\end{equation}
\begin{equation}\label{q2}
	\bar D_{U}N=-\bar A_N U+\bar{\nabla}^{l}_U N+\bar{D^{s}}(U,N)	
\end{equation}
\begin{equation}\label{q3}
	\bar D_{U}W=-\bar A_W U+\bar{\nabla}^{s}_U W+\bar{D^{l}}(U,W)	
\end{equation}

for any $U,V\in\Gamma(TM)$, $N\in\Gamma(ltr(TM))$ and $W\in\Gamma(S(TM^{\perp}))$, where $\{D_U V,\bar A_N U,\bar A_W U\} \in\Gamma(TM)$ and $\bar{\nabla}^{l}$ and $\bar{\nabla}^{s}$ are linear connections on $ltr(TM)$ and $S(TM^{\perp})$ respectively. Both $\bar A_N$ and $\bar A_W$ are linear operators on $\Gamma(TM)$. From $(\ref{a1})$,$(\ref{q1})$,$(\ref{q2})$,$(\ref{q3})$, we obtain
\begin{equation*}\label{q4}
	D_U V= \nabla_U V +\pi(V)fU
\end{equation*}
\begin{equation}\label{q5}
	\bar{h^{l}}(U,V)= h^{l}(U,V)+\pi(V)w_{l}U
\end{equation}
\begin{equation}\label{q6}
	\bar{h^{s}}(U,V)= h^{s}(U,V)+\pi(V)w_{s}U
\end{equation}
\begin{equation*}\label{q7}
	\bar A_N U= A_N U-\pi(N)fU
\end{equation*}
\begin{equation*}\label{q8}
	\bar \nabla_U^{l} N=\nabla_U^{l}N+\pi(N)w_{l}U
\end{equation*}
\begin{equation*}\label{q9}
	\bar{D^{s}}(U,N)=D^{s}(U,N)+\pi(N)w_{s}U
\end{equation*}
\begin{equation*}\label{q10}
	\bar A_W U= A_W U-\pi(W)fU
\end{equation*}
\begin{equation*}\label{q11}
	\bar \nabla_U^{s} W=\nabla_U^{s}W+\pi(W)w_{s}U
\end{equation*}
\begin{equation*}\label{q12}
	\bar{D^{l}}(U,W)=D^{l}(U,W)+\pi(W)w_{l}U
\end{equation*}
From (\ref{q4}) we get,
\begin{equation*}\label{q13}
	(D_U g)(V,Z)=g(h(U,V),Z)+g(h(U,Z),V)-\pi(V)g(fU,Z)-\pi(Z)g(fU,V)
\end{equation*}
On the other hand, the torsion tensor of the induced connection $D$ is 
\begin{equation*}\label{q14}
	T^{D}(U,V)= \pi(V)fU-\pi(U)fV
\end{equation*}
\begin{equation}\label{q15}
	\bar{g}(\bar h^{s}(U,V),W)+\bar{g}(V,\bar D^{l}(U,W))=\bar{g}(\bar A_W U,V)+\pi(W)\bar{g}(fU,V)+\pi(V)\bar{g}(w_s U, W)+\pi(W)\bar{g}(V,w_l U)
\end{equation}
\begin{equation}\label{q16}
	\bar{g}(\bar D^{s}(U,N),W)=\bar{g}(\bar A_W U,N)+\pi(W)\bar{g}(fU,N)+\pi(N)\bar{g}(w_s U,W)
\end{equation}\\
\begin{proposition}
	Let $M$ be a lightlike submanifold of a metallic semi-Riemannian manifold $\bar{M}$ with a quarter symmetric non-metric connection $\bar{D}$. Then the induced connection $D$ on the lightlike submanifold $M$ is also a quarter symmetric non-metric connection.
\end{proposition}
Let $J$ be the projection of $TM$ on $S(TM)$, then any $U\in\Gamma(TM)$, can be written as $U=JU+\Sigma_{i=1}^{r} \eta_{i}(U)\xi_{i}$, 
\begin{equation}\label{z1}
	\eta_{i}(U)=g(U,N_{i})
\end{equation}, where $\{\xi_{i}\}_{i=1}^{r}$ is a basis for $Rad(TM)$.Therefore,\\
\begin{equation}\label{q17}
	D_U JV= D_U^{*}JV+\bar h^{*}(U,JV)
\end{equation}
\begin{equation}\label{q18}
	D_U \xi=-\bar A_\xi ^{*} U+ \bar \nabla_U^{*t}\xi
\end{equation}
For any $U,V\in\Gamma(TM)$, where $\{D_U^{*}JV,\bar A_\xi ^{*} U\}\in\Gamma(S(TM))$ and $\{\bar h^{*}(U,JV),\bar \nabla_U^{*t}\xi\}\in\Gamma(Rad(TM))$. From $(\ref{q17}),(\ref{q18})$ we obtain\\
\begin{equation*}\label{q19}
	D_U^{*}JV=\nabla_U^{*}JV+\pi(JV)JfU
\end{equation*}
\begin{equation}\label{q20}
	\bar h^{*}(U,JV)=h^{*}(U,JV)+\pi(JV)\Sigma_{i=1}^{r} \eta_{i}(fU)\xi_{i}
\end{equation}
and 
\begin{equation}\label{q21}
	\bar A_\xi ^{*} U=A_{\xi}^{*}U-\pi(\xi)JfU
\end{equation}
\begin{equation*}\label{q22}
	\bar \nabla_U^{*t}\xi=  \nabla_U^{*t}\xi+\pi(\xi)\eta(fU)\xi
\end{equation*}
From (\ref{q5}),(\ref{q7}),(\ref{q20}),(\ref{q21}),we have\\
\begin{equation*}\label{q23}
	g(\bar h^{l}(U,JV),\xi)=g(\bar A_\xi ^{*} U,JV)+\pi(\xi)g(JfU,JV)+\pi(JV)g(w_{l}U,\xi)
\end{equation*}
\begin{equation*}\label{q24}
	g(\bar h^{*}(U,JV),N)=g(\bar A_N U,JV)+\pi(N)g(fU,JV)+\pi(JV)\eta(fU)
\end{equation*}
\begin{equation*}\label{q25}
	\bar{g}(\bar h^{l}(U,\xi),\xi)=\pi(\xi)\bar{g}(w_{l}U,\xi),\quad \bar A_\xi ^{*} \xi=-\pi(\xi)f\xi
\end{equation*}
\section{INVARIANT LIGHTLIKE SUBMANIFOLDS}\label{sec3}\
This section provides an illustration that elucidates the configuration of an invariant lightlike submanifold of metallic semi-Riemannian manifold. Subsequently, it examines its geometry when the manifold is endowed with a quarter symmetric non-metric connection.
\begin{definition}\cite{GSRM5}
	Let $M$ be a lightlike submanifold of metallic semi-Riemannian manifold $\bar{M}$. If
	\begin{equation}\label{21}
		\bar{J}Rad(TM)=Rad(TM),\quad \bar{J}S(TM)=S(TM)	
	\end{equation}
	
	then, $M$  is said to be invariant lightlike submanifold of $\bar{M}$ metallic semi-Riemannian manifold.\\
\end{definition}
\begin{proposition}\cite{GSRM5}
	Let $M$ be an invariant lightlike submanifold of a metallic semi-Riemannian manifold $(\bar{M},\bar{J})$. Then, the distribution $ltr(TM)$ is invariant with respect to $\bar{J}$.
\end{proposition}
\begin{example}
	Consider a metallic semi-Riemannian manifold $\bar{M}= (R_1^{5},\bar{g})$ of signature $(-,+,+,+,+)$ with respect to the basis $\{\partial y_1,\partial y_2,\partial y_3,\partial y_4,\partial y_5\}$, where the metallic structure $\bar{J}$ is defined by $\bar{J}(y_1,y_2,y_3,y_4,y_5)=(\omega y_1,(p-\omega) y_2,(p-\omega) y_3,\omega y_4,(p-\omega) y_5)$.\\
	
	Let $M$ be a submanifold of $(R_1^{5},\bar{J},\bar{g})$ given by\\
	\begin{equation*}
		y_1=0,\quad y_2=\frac{1}{2}(\sqrt{3}t_4+ t_2), 
	\end{equation*}
	\begin{equation*}
		y_3=\frac{1}{2}(-\sqrt{3}t_2+ t_4),\quad y_4=t_1,\quad y_5=t_4,
	\end{equation*}
	Then $TM$ is spanned by $\{D_1,D_2,D_3\}$, where
	\begin{equation*}
		D_1= \frac{1}{2}(-\sqrt{3}\partial y_3+\partial y _2),
	\end{equation*}
	\begin{equation*}
		D_2=\partial y_4,
	\end{equation*}
	\begin{equation*}
		D_3=\partial y_5+ \frac{1}{2}(\sqrt{3}\partial y_2+\partial y _3)
	\end{equation*}\\
	Hence $M$ is a 1-lightlike submanifold of $R_1^{5}$ with \\
	\begin{equation*}
		Rad(TM)=Span\{D_3\}\quad and\quad S(TM)=Span\{D_1,D_2\}
	\end{equation*}
	Therefore
	\begin{equation*}
		\bar{J}D_3= (p-\omega)D_3\in\Gamma(Rad(TM)),\quad \bar{J}D_1=(p-\omega) D_1\in\Gamma(S(TM)),\quad \bar{J}D_2=\omega D_2\in\Gamma(S(TM)),
	\end{equation*}\\
	which means that $S(TM)$ and $Rad(TM)$ is invariant with respect to $\bar{J}$. Also, for $N\in\Gamma(ltr(TM))$ and  $W\in\Gamma(S(TM^{\perp}))$, where\\
	\begin{equation*}
		N=\frac{1}{2}\{-\partial y_5+ \frac{1}{2}\partial y_3+\frac{\sqrt{3}}{2} \partial y_2\},\quad W=\partial y_1
	\end{equation*}
	It follows that  $ltr(TM)$ and $S(TM^{\perp})$ are invariant distributions with respect to $\bar{J}$. Thus, $M$ is an invariant lightlike submanifold of $\bar{M}$.
\end{example}

\begin{theorem}
	For an invariant lightlike submanifold $M$ of a metallic semi-Riemannian $\bar{M}$ with a quarter symmetric non-metric connection $\bar{D}$, the radical distribution is integrable  if and only if 
	\[
	\bar A^{*}_{\bar{J}U}V-p \bar A^{*}_{U}V= \bar A^{*}_{\bar{J}V}U-p\bar A^{*}_{V}{U}
	\]
	for $U,V\in{\Gamma(Rad\;TM)}, Z\in\Gamma(S(TM))$.\\
\end{theorem}
\begin{proof}

	$g([U,V],Z)=\bar{g}(\bar D_U V-\pi(V)\bar{J}U-\bar D_V U+ \pi(U)\bar{J}V,Z)$\\
	
	\qquad $=\frac{1}{q}[\bar{g}(\bar D_U \bar{J}V,\bar{J}Z)-{p}\bar{g}(\bar D_U V,\bar{J}Z)-\bar{g}(\bar D_V \bar{J}U,\bar{J}Z)+{p}\bar{g}(\bar D_V U,\bar{J}Z)]$\\
	
	\qquad $=\frac{1}{q}[\bar{g}(-\bar A_{\bar J V}^{*}U+\bar{\nabla}_U^{*t}\bar{J} V,\bar{J}Z)-{p}\bar{g}(-\bar A_{ V}^{*}U+\bar{\nabla}_U^{*t} V,\bar{J}Z)-\bar{g}(-\bar A_{\bar J U}^{*}V+\bar{\nabla}_V^{*t}\bar{J} U,\bar{J}Z)+$\\
	
	\qquad$ {p}\bar{g}(-\bar A_{U}^{*}V+\bar{\nabla}_V^{*t}U,\bar{J}Z)]$\\
	
	\qquad$=\bar{g}(\bar A^{*}_{\bar{J}U}V-p \bar A^{*}_{U}V- \bar A^{*}_{\bar{J}V}U=p\bar A^{*}_{V}{U},\bar{J}Z)=0$\\
	
	Using the integrability of $Rad(TM)$ along with equation (\ref{a1}), (\ref{a2}), (\ref{q1}) and (\ref{q18}) the result follows. 
	\end{proof}
	
	\begin{theorem}
		If $M$ an invariant lightlike submanifold of metallic semi-Riemannian manifold $\bar{M}$ with a quarter symmetric non-metric connection $\bar{D}$, then the screen distribution is integrable iff\\
		\[\bar{h}^{*}(V,\bar{J}U)+p \bar{h}^{*}(U,V)= \bar{h}^{*}(U,\bar{J}V)+p\bar{h}^{*}(V,U)\]\\
		for all $U,V\in\Gamma (S(TM)), N\in\Gamma (ltr(TM))$.
		
		\begin{proof}
			$S(TM)$ is integrable iff $g([V,U],N)=0$ for all $U,V\in\Gamma (S(TM)), N\in\Gamma (ltr(TM))$.\\
			
			$0=\bar{g}(\bar{D}_V U,N)-\bar{g}(\bar{D}_U V,N)$\\
			
			\qquad$=\frac{1}{q}[\bar{g}(\bar D_V \bar{J}U,\bar{J}N)+p \pi(U)\bar{g}(\bar{J}V,\bar{J}N)+ q\pi(U)\bar{g}(V,\bar{J}N)-\pi(\bar{J}U)\bar{g}(\bar{J}V,\bar{J}N)-$\\
			
			\qquad$p\bar{g}(\bar D_V U,\bar{J}N)]-\frac{1}{q}[\bar{g}(\bar D_U \bar{J}V,\bar{J}N)+p \pi(V)\bar{g}(\bar{J}U,\bar{J}N)+ q\pi(V)\bar{g}(U,\bar{J}N)-$\\
			
			\qquad$\pi(\bar{J}V)\bar{g}(\bar{J}U,\bar{J}N)-p\bar{g}(\bar D_U V,\bar{J}N)]$\\

			\qquad$=\frac{1}{q}\bar{g}(D_V^{*} \bar{J}U,\bar{J}N)+\frac{1}{q}\bar{g}(\bar{h^{*}}(V,\bar{J}U),\bar{J}N)-\frac{p}{q}\bar{g}(D_V^{*}U,\bar{J}N)-\frac{p}{q}\bar{g}(\bar{h^{*}}(V,U),\bar{J}N)-$\\
			
			\qquad$\frac{1}{q}\bar{g}(D_U^{*}\bar{J}V,\bar{J}N)-\frac{1}{q}\bar{g}(\bar{h^{*}}(U,\bar{J}V),\bar{J}N)+\frac{P}{q}\bar{g}(D_U^{*} V,\bar{J}N)+\frac{p}{q}\bar{g}(\bar{h^{*}}(U,V),\bar{J}N)$\\
			
			Therefore, the equations (\ref{c}), (\ref{a2}) and (\ref{q17}) lead to the desired result.

		\end{proof}
	\end{theorem}
	
	\begin{theorem}
		Let $M$ be an invariant lightlike submanifolds of metallic semi-Riemannian manifold $\bar{M}$ with a quarter symmetric non-metric connection $\bar{D}$. The radical distribution defines a totally geodesic foliation iff\\
		\[\bar A_{\bar{J}V}^{*} U=-p \bar A_V ^{*} U\]
		for all $U,V\in\Gamma (Rad(TM)), Z\in\Gamma (S(TM))$.
		\begin{proof}
			From the concept of quarter symmetric non-metric connection and the metallic structure on $\bar{M}$, we derive\\

			$g(D_U V,N)=\frac{1}{q}[\bar{g}(\bar{J}\bar D_U V,\bar{J}N)-p\bar{g}(\bar{J}\bar D_U V,N)]$\\
			
			Using equations (\ref{a1}), (\ref{a2}), (\ref{q1}) and (\ref{q18}) we get;\\
			
			$g(D_U V,N)=\frac{1}{q}\bar{g}(\bar D_U \bar{J}V,\bar{J}N)+\frac{p}{q} \pi(V)\bar{g}(\bar{J}U,\bar{J}N)+ \pi(V)\bar{g}(U,\bar{J}N)-\frac{1}{q}\pi(\bar{J}V)\bar{g}(\bar{J}U,\bar{J}N)-$\\
			
			\qquad$\frac{p}{q}\bar{g}(\bar D_U V,\bar{J}N)$\\
			
			\qquad$=\frac{1}{q}[\bar{g}( D_U \bar{J}V,\bar{J}N)-p\bar{g}( D_U V,\bar{J}N)]$\\
			
			\qquad$=\frac{1}{q}\bar{g}(-\bar A_{\bar J V}^{*}U,\bar{J}N)+\frac{1}{q} \bar{g}(\bar{\nabla}_U^{*t}\bar{J} V,\bar{J}N)+\frac{p}{q}\bar{g}(\bar A_{ V}^{*}U,\bar{J}N)-\frac{p}{q}\bar{g}(\bar{\nabla}_U^{*t} V,\bar{J}N)=0$\\
			
			Therefore, the result follows from the hypothesis.          
			
		\end{proof}
	\end{theorem}

	\begin{theorem}
		Let $M$ be an invariant lightlike submanifold of metallic semi-Riemannian manifold $\bar{M}$ with a quarter symmetric non-metric connection $\bar{D}$.Then the screen distribution defines a totally geodesic foliation  iff\\
		\[\bar h^{*}(U,\bar{J}V)=p \bar h^{*}(U,V)\] 
		for any $U,V\in\Gamma(S(TM)), N\in\Gamma(ltr(TM))$.
		
		\begin{proof}
			$S(TM)$ defines totally geodesic foliation iff $D_U V\in\Gamma(S(TM))$\\
			
			$g(D_U V,N)=\frac{1}{q}[\bar{g}(\bar D_U \bar{J}V+p\pi(V)\bar{J}U+q\pi(V)U-\pi(\bar{J}V)\bar{J}U,\bar{J}N)]-\frac{p}{q}\bar{g}(\bar D_U V,\bar{J}N)$\\
			
			Following (\ref{q1}) and (\ref{q17}) we have\\
			
			\qquad$=\frac{1}{q}\bar{g}(D_U \bar{J}V,\bar{J}N)-\frac{p}{q}\bar{g}(D_U V,\bar{J}N)+\frac{1}{q}\bar{g}(\bar h^{l}(U,\bar{J}V),\bar{J}N)-\frac{p}{q}\bar{g}(\bar h^{l}(U,V),\bar{J}N)+$\\
			
			\qquad$\frac{1}{q}\bar{g}(\bar h^{s}(U,\bar{J}V),\bar{J}N)-\frac{p}{q}\bar{g}(\bar h^{s}(U,V),\bar{J}N)$\\
			
			\qquad$=\frac{1}{q}\bar{g}(D_U^{*}\bar{J}V,\bar{J}N)+\frac{1}{q}\bar{g}(\bar h^{*}(U,\bar{J}V),\bar{J}N)-\frac{p}{q}\bar{g}(D_U^{*}V,\bar{J}N)-\frac{p}{q}\bar{g}(\bar h^{*}(U,V),\bar{J}N)$\\
			
			\qquad$=\frac{1}{q}[\bar{g}(\bar h^{*}(U,\bar{J}V)-p\bar h^{*}(U,V),\bar{J}N)]=0$\\
			
			Therefore, the assertion follows from the hypothesis.
		\end{proof}
	\end{theorem}
\vspace{1in}

\section{SCREEN SEMI-INVARIANT LIGHTLIKE SUBMANIFOLDS WITH QUARTER SYMMETRIC NON-METRIC CONNECTION}\label{sec4}\

The structure of a screen semi-invariant submanifold has been expounded upon with an example  and further analysed for the metallic semi-Riemannian manifold with a quarter symmetric non-metric connection. 

\begin{definition}\label{d1}
	Let $M$ be a lightlike submanifold of a metallic semi-Riemannian manifold $\bar{M}$.Then  $M$ is said to be a screen semi-invariant lightlike submanifold of $\bar{M}$ if the following conditions are satisfied:\\
	(1) There exists a non-null distribution $B\subseteq{S(TM)}$ such that\\
	\begin{equation}\label{S1}
		S(TM)=B\oplus{B^{\perp}},\bar{J}(B)=B,\bar{J}(B^{\perp})\subseteq{S(TM^{\perp})},B\cap{B^{\perp}}=\{0\}
	\end{equation}\\
	where $B^{\perp}$ is orthogonal complementary to $B$ in $S(TM)$.\\
	
	(2) $Rad(TM)$ is invariant with respect to $\bar{J}$ i.e. $\bar{J}(Rad(TM))=Rad(TM)$.\\
	Then 
	\begin{equation}\label{S2}
		\bar{J}(ltr(TM))=ltr(TM)
	\end{equation}
	\begin{equation}\label{S3}
		TM=B^{'}\oplus{B^{\perp}},\quad B^{'}=B\perp{Rad(TM)}
	\end{equation}
	It follows that $B^{'}$ is also invariant with respect to $\bar{J}$.Thus the orthogonal complement to $\bar{J}(B^{\perp})$ in $S(TM^{\perp})$ is indicated by  $B_{o}$. Then we obtain\\
	\begin{equation}\label{S4}
		tr(TM)=ltr(TM)\perp{\bar{J}(B^{\perp})}\perp{B_{o}}
	\end{equation}
\end{definition}

\begin{proposition}
	Let $M$ be a screen semi-invariant lightlike submanifold of a metallic semi-Riemannian manifold $(\bar{M},\bar{g},\bar{J})$. Then, $M$ is an invariant lightlike submanifold of $\bar{M}$ if and only if $B^{\perp}=\{0\}$.
\end{proposition}
\begin{proof}
	Since $M$ is a invariant lightlike submanifold of $\bar{M}$, therefore $\bar{J}(TM)=(TM)$ .Hence $B^{\perp}=\{0\}$. Similarly the converse hold. 
\end{proof}

\begin{proposition}
	If a screen semi-invariant lightlike submanifold of a metallic semi-Riemannian manifold $\bar{M}$ is a isotropic or totally lightlike , then it is an invariant lightlike submanifold of $\bar{M}$.
\end{proposition}

\begin{lemma}
	Let $M$ be a screen semi-invariant lightlike submanifold of a metallic semi-Riemannian manifold $\bar{M}$. Then\\
	
	$f^{2}U=pfU+qU-BwU$,\quad $pwU-CwU=wfU$\\
	
	$fBN=pBN-BCN$, \quad $C^{2}N=pCN+qN-wBN$\\
	
	$g(fU,V)-g(U,fV)=g(U,wV)-g(wU,V)$\\
	
	$g(fU,fV)=pg(fU,V)+qg(U,V)+pg(wU,V)-g(fU,wV)-g(wU,fV)-g(wU,wV)$\\
	
	for any $U,V\in\Gamma(TM)$ and $ N\in\Gamma(tr(TM))$.
\end{lemma}

\begin{proof}
	
	$\bar{J}U=fU+wU$ for $U\in\Gamma(TM)$\\
	
	Applying $\bar{J}$ on both sides and using $(\ref{a})$, we derive\\
	
	$pfU+pwU+qU=f^{2}U+ wfU + BwU + CwU$\\
	
	Taking tangential and transversal parts of the above equation,\\
	
	$f^{2}U=pfU+qU-BwU$ and $pwU-CwU=wfU$\\
	
	The application of $\bar{J}$ on $(\ref{t2})$ along with equation $(\ref{a})$ and then on the comparison of  tangential and transversal parts, we obtain\\
	
	$fBN=pBN-BCN$ and $C^{2}N=pCN+qN-wBN$\\
	
	Using equations $(\ref{a})$,$(\ref{b})$,$(\ref{c})$, we get\\
	
	$g(fU,V)-g(U,fV)=g(U,wV)-g(wU,V)$\\
	
	$g(fU,fV)=pg(fU,V)+qg(U,V)+pg(wU,V)-g(fU,wV)-g(wU,fV)-g(wU,wV)$
	
\end{proof}
\begin{theorem}
	Let $M$ be a screen semi-invariant lightlike submanifold of a metallic semi-Riemannian manifold $\bar{M}$. Then, $f$ is a metallic structure on $B^{'}$.
\end{theorem}
\begin{proof}
	For a screen semi-invariant lightlike submanifold, we have $wU=0$ for any $U\in\Gamma(B^{'})$. Then it follows from the above lemma that $f$ become a metallic structure on $B^{'}$.
\end{proof}

Inspired by \cite{GSRM9}, we present the following example:
\begin{example}
	Consider an $9$-dimensional semi-Euclidean space $(\bar{M}=\mathbb{R}_{2}^{9},\bar{g})$ with signature $(-,-,+,+,+,+,+,+,+)$. Let $(z_1,z_2,z_3,z_4,z_5,z_6,z_7,z_8,z_9)$ be the standard coordinate system of $\bar{M}$.Then by setting   
	\begin{equation*}
		\bar{J}(z_1,z_2,z_3,z_4,z_5,z_6,z_7,z_8,z_9)=((\sigma z_1,(p-\sigma){z_2},\sigma{z_3},\sigma{z_4},\sigma{z_5},\sigma{z_6},\sigma{z_7},\sigma{z_8},(p-\sigma){z_9})
	\end{equation*} 
	we have $\bar{J}^{2}=p\bar{J}+qI$ which shows that $\bar{J}$ is a metallic structure on $\bar{M}$.\\
	
	Consider a submanifold $M$ of $\mathbb{R}_{2}^{9}$ is given by the equations\\  
	
	\begin{equation*}
		z_1=-y_2+y_1,\quad z_2= \frac{\sigma}{\sqrt{q}} y_4
	\end{equation*}
	\begin{equation*}
		z_3=y_3+y_5,  \quad   z_4=y_2+y_3
	\end{equation*}
	\begin{equation*}
		z_5=y_4,\quad z_6=y_3-y_5
	\end{equation*}
	\begin{equation*}
		z_7=y_2+y_1,\quad z_8= -y_2+y_3,\quad z_9=0
	\end{equation*}
	
	Here $TM$ is spanned by $\{D_1,D_2,D_3,D_4,D_5\}$, where\\
	\begin{equation*}
		D_1=-\partial{z_1}+\partial{z_4}+\partial{z_7}-\partial{z_8},
		\quad D_2=\partial{z_3}+\partial{z_4}+\partial{z_6}+\partial{z_8}
	\end{equation*}
	\begin{equation*}
		D_3=\partial{z_1}+\partial{z_7}, \quad D_4=\partial{z_3}-\partial{z_6},\quad D_5=\frac{\sigma}{\sqrt{q}} \partial{z_2}+\partial{z_5}
	\end{equation*}
	
	We see that $M$ is a $2$-lightlike submanifold with $Rad(TM)$=Span $\{D_1,D_2 \}$. Further, $S(TM)$ and $S(TM^{\perp})$ are spanned by $\{D_3,D_4,D_5\}$ and $\{W\}$ respectively, where\\
	
	\begin{equation*}
		W=-\sqrt{q}{\partial{z_2}}+\sigma \partial{z_5}\\
	\end{equation*}
	
	The lightlike transversal vector bundle $ltr(TM)$ is spanned by\\
	\begin{equation*}
		N_1=\frac{1}{4}(-\partial{z_1}-\partial{z_4}+\partial{z_7}+\partial{z_8})
	\end{equation*}
	\begin{equation*}
		N_2=\frac{1}{4}(\partial{z_3}-\partial{z_4}+\partial{z_6}-\partial{z_8})
	\end{equation*}\\
	Hence, $B=Span\{D_3,D_4\}$, $B^{\perp}=Span\{D_5\}$, $B_{o}=\{0\}$ and $B^{'}=Span\{D_1,D_2,D_3,D_4\}$.\\
	
	Also, $\bar{J}(Rad(TM))=Rad(TM)$, $\bar{J}(B)=B$, $\bar{J}(ltr(TM))=ltr(TM)$ and $\bar{J}(B^{\perp})=S(TM^{\perp})$. 
	
	Therefore, $M$ becomes a screen semi-invariant lightlike submanifold of the metallic semi-Riemannian manifold $\bar{M}$. 
	
\end{example}

\begin{lemma}
	Let $(M,g,S(TM),S(TM^{\perp}))$ be a screen semi-invariant lightlike submanifold of the metallic semi-Riemannian manifold $(\bar{M},\bar{g},\bar{J})$ with a quarter symmetric non-metric connection $\bar{D}$.Then
	\begin{equation*}
		g(\bar{h}(U,V),\xi)=g(\bar{A}_{\xi}^{*}U,V)+ \pi(\xi)g(JfU,V),
	\end{equation*}
	for any $\xi\in\Gamma(Rad(TM))$,\quad $U\in\Gamma(B)$,\quad $V\in\Gamma(B^{'})$.
\end{lemma}
\begin{proof}
	For any $\xi\in\Gamma(Rad(TM))$,\quad $U\in\Gamma(B)$,\quad $V\in\Gamma(B^{'})$\\
	
	$g(\bar{h}(U,V),\xi)= g(h(U,V)+\pi(V) wU,\xi)=g(A^{*}_{\xi}U,V)=g(\bar{A}_{\xi}^{*}U,V)+ \pi(\xi)g(JfU,V)$
\end{proof}
\begin{definition}
	Let $(M,g,S(TM),S(TM^{\perp}))$ be a screen semi-invariant submanifold of metallic semi-Riemannian manifold $(\bar{M},\bar{g},\bar{J})$. Then, the distribution $B$ is integrable if and only if $[U,V]\in\Gamma(B)$ for any $U,V\in\Gamma(B)$.
\end{definition}

\begin{theorem}
	Let $M$ be a screen semi-invariant lightlike submanifold of metallic semi-Riemannian manifold $\bar{M}$ with a quarter symmetric non-metric connection $\bar{D}$.Then the radical distribution is integrable  iff\\
	
	$(i) \bar A_{\bar{J}E}^{*}E^{'}+p\bar A_{E^{'}}^{*}E=\bar A_{\bar{J}E^{'}}^{*}E+p\bar A_{E}^{*}E^{'}$ for all $E,E^{'}\in\Gamma(Rad(TM)),  U\in\Gamma(B)$.\\
	
	$(ii)\bar h^{s}(E,\bar{J}E^{'})=\bar h^{s}(E^{'},\bar{J}E)$ for all $E,E^{'}\in\Gamma(Rad(TM)), Z\in\Gamma(B^{\perp})$.\\

	\begin{proof}
		$Rad(TM)$is integrable iff 
		\begin{equation}{\label{th1}}
			g([E,E^{'}],Z)=0,\quad g([E,E^{'}],U)=0
		\end{equation} 
		for any $E,E^{'}\in\Gamma(Rad(TM)),Z\in\Gamma(B^{\perp}),U\in\Gamma(B)$\\
		
		From equations (\ref{th1}), (\ref{a1}), (\ref{a2}), (\ref{c}), (\ref{q1}) and (\ref{q18}), we get\\
		
		$0=\frac{1}{q}[\bar{g}(\bar{D}_E \bar{J}E^{'}+p\pi(E^{'})\bar{J}E+q\pi(E^{'})E-\pi(\bar{J}E^{'})\bar{J}E,\bar{J}U)-p\bar{g}(\bar{D}_E E^{'},\bar{J}U)]-$\\
		
		\qquad $\frac{1}{q}[\bar{g}(\bar{D}_{E^{'}} \bar{J}E+p\pi(E)\bar{J}E^{'}+q\pi(E)E^{'}-\pi(\bar{J}E)\bar{J}E^{'},\bar{J}U)-p\bar{g}(\bar{D}_{E^{'}} E,\bar{J}U)]$\\
		
		\qquad $=\frac{1}{q}[\bar{g}(-\bar A_{\bar{J}E^{'}}^{*}E,\bar{J}U)+\bar{g}(\bar{\nabla}_E ^{*t}\bar{J}E^{'},\bar{J}U)-p\bar{g}(-\bar A_{E^{'}}^{*}E,\bar{J}U)-p\bar{g}(\bar{\nabla}_E ^{*t}E^{'},\bar{J}U)]-$\\
		
		\qquad $\frac{1}{q}[\bar{g}(-\bar A_{\bar{J}E}^{*}E^{'},\bar{J}U)+\bar{g}(\bar{\nabla}_{E^{'}} ^{*t} \bar{J}E,\bar{J}U)-p\bar{g}(-\bar A_{E}^{*}E^{'},\bar{J}U)-p\bar{g}(\bar{\nabla}_{E^{'}} ^{*t}E,\bar{J}U)]$\\
		
		\qquad $=\frac{1}{q}\bar{g}(\bar A_{\bar{J}E}^{*}E^{'}+p\bar A_{E^{'}}^{*}E-\bar A_{\bar{J}E^{'}}^{*}E-p\bar A_{E}^{*}E^{'},\bar{J}U)$\\
		
		Hence (i) follows.\\
		
		$g([E,E^{'}],Z)=\bar{g}(\bar D_E {E^{'}}-\pi(E^{'})\bar{J}E-\bar{D}_{E^{'}} E+\pi(E)\bar{J}E^{'},Z)$\\

		Using equations, (\ref{a1}), (\ref{a2}) and (\ref{q1}) \\ 
		
		$g([E,E^{'}],Z)=\frac{1}{q}\bar{g}(D_{E}{\bar{J}E^{'}},\bar{J}Z)+\frac{1}{q}\bar{g}(\bar{h^{l}}(E,\bar{J}E^{'}),\bar{J}Z)+\frac{1}{q}\bar{g}(\bar{h^{s}}(E,\bar{J}E^{'}),\bar{J}Z)-$\\
		
		\qquad$\frac{p}{q}\bar{g}(D_{E}E^{'},\bar{J}Z)-\frac{p}{q}\bar{g}(\bar{h^{l}}(E,E^{'}),\bar{J}Z)-\frac{p}{q}\bar{g}(\bar{h^{s}}(E,E^{'}),\bar{J}Z)-\frac{1}{q}\bar{g}(D_{E^{'}}{\bar{J}E},\bar{J}Z)-$\\
		
		\qquad $\frac{1}{q}\bar{g}(\bar{h^{l}}(E^{'},\bar{J}E),\bar{J}Z)-\frac{1}{q}\bar{g}(\bar{h^{s}}(E^{'},\bar{J}E),\bar{J}Z)+\frac{p}{q}\bar{g}(D_{E^{'}}{E},\bar{J}Z)+\frac{p}{q}\bar{g}(\bar{h^{l}}(E^{'},E),\bar{J}Z)+$\\
		
		\qquad $\frac{p}{q}\bar{g}(\bar{h^{s}}(E^{'},E),\bar{J}Z)$\\ 
		
		From equations,(\ref{S1}), (\ref{q6}) and (\ref{th1})\\
		
		$g([E,E^{'}],Z)=\frac{1}{q}[\bar{g}(\bar{h^{s}}(E,\bar{J}E^{'})-p\pi(E^{'})w_s E-\bar{h^{s}}(E^{'},\bar{J}E)+p\pi(E)w_s E^{'},\bar{J}Z)]=0$\\
		
		Using definition of screen semi-invariant submanifold, $w_s E=0,\quad w_s E^{'}=0$\\
		
		$\bar h^{s}(E,\bar{J}E^{'})=\bar h^{s}(E^{'},\bar{J}E)$ which proves $(ii)$.\\

	\end{proof}
\end{theorem}
\begin{theorem}
	Let $M$ be a screen semi-invariant lightlike submanifold of metallic semi-Riemannian manifold $\bar{M}$ with a quarter symmetric non-metric connection $\bar{D}$.The necessary and sufficient condition for $B$ to be integrable is that \\
	
	$\bar h^{s}(U,\bar{J}V)=\bar h^{s}(V,\bar{J}U)$ for any  $U,V\in\Gamma(B),Z\in\Gamma(B^{\perp})$.\\
	
	$\bar h^{*}(U,\bar{J}V)+p\bar h^{*}(V,U)=\bar h^{*}(V,\bar{J}U)+p\bar h^{*}(U,V)$ \\
	
	for $U,V\in\Gamma(B),N\in\Gamma(ltr(TM))$\\
	
	\begin{proof}
		$B$ is integrable iff
		\begin{equation}\label{th2}
			g([U,V],Z)=0, \quad g([U,V],N)=0 
		\end{equation}\\
		$g([U,V],Z)= \bar{g}(\bar D_U V-\pi(V)\bar{J}U-\bar D_V U+ \pi(U)\bar{J}V,Z)=\bar{g}(\bar D_U V,Z)-\bar{g}(\bar D_V U,Z)$\\
		
		Using equations (\ref{c}) and (\ref{a2})\\
		
		$g([U,V],Z)=\frac{1}{q}\bar{g}(\bar{J}\bar D_U V,\bar{J}Z)-\frac{p}{q}\bar{g}(\bar D_U V,\bar{J}Z)-\frac{1}{q}\bar{g}(\bar{J}\bar D_V U,\bar{J}Z)+\frac{p}{q}\bar{g}(\bar D_V U,\bar{J}Z)$\\
		
		Following equations (\ref{q1}), (\ref{th2}) and (\ref{q6})\\

		\qquad$\frac{1}{q}[\bar{g}(\bar h^{s}(U,\bar J{V})-p\pi(V)w_s U-\bar h^{s}(V,\bar J{U})+p\pi(U)w_s V,\bar{J}Z)]=0$\\
		
		The result follows from (\ref{S1}) and (\ref{t1}).\\

		Also, $g([U,V],N)=\frac{1}{q}[\bar{g}(\bar D_U \bar{J}V+p\pi(V)\bar{J}U+q\pi(V)U-\pi(\bar{J}V)\bar{J}U,\bar{J}N)-p\bar{g}(\bar D_U V,\bar{J}N)]-$\\
		
		\qquad $\frac{1}{q}[\bar{g}(\bar D_V \bar{J}U+p\pi(U)\bar{J}V+q\pi(U)V-\pi(\bar{J}U)\bar{J}V,\bar{J}N)-p\bar{g}(\bar D_V U,\bar{J}N)]$\\
		
		Since $M$ is a screen semi-invariant submanifold of $\bar{M}$. Therefore, on using\\
		
		equations $(\ref{q1}),(\ref{q17}),(\ref{q20}),(\ref{th2})$, we derive\\
		
		$g([U,V],N) =\frac{1}{q}[\bar{g}(\bar h^{*}(U,\bar{J}V)+ph^{*}(V,U)+\pi(U)\eta(fV)\xi-\bar h^{*}(V,\bar{J}U)-\pi(V)\eta(fU)\xi-$\\
		
		\qquad$ph^{*}(U,V),\bar{J}N)]=0$\\

		Hence the result follows (\ref{S1}) and (\ref{z1}).

	\end{proof}
	
\end{theorem}

\begin{theorem}
	The distribution $B^{'}$ of a screen semi-invariant lightlike submanifold of metallic semi-Riemannian manifold $\bar{M}$ equipped with a quarter symmetric non-metric connection is integrable if and only if \\
	\[\bar h^{s}(U,\bar{J}V)=\bar h^{s}(V,\bar{J}U)\]
	for any $U,V\in\Gamma(B),\quad Z\in\Gamma(B^{\perp})$.
	\begin{proof}
		$B^{'}$ is integrable iff \\
		\begin{equation}\label{th4}
			g([U,V],Z)=0
		\end{equation}\\
		Using equations (\ref{c}), (\ref{a1}) and (\ref{a2})\\

		$g([U,V],Z)= \bar{g}(\bar{\nabla}_U V - \bar{\nabla}_V U,Z)=\bar{g}(\bar D_U V,Z)-\bar{g}(\bar D_V U,Z)$\\

		\qquad$=\frac{1}{q}[\bar{g}(\bar D_U \bar{J}V+p\pi(V)\bar{J}U+q\pi(V)U-\pi(\bar{J}V)\bar{J}U,\bar{J}Z)]-\frac{p}{q}\bar{g}(\bar D_U V,\bar{J}Z)-$\\
		
		\qquad$\frac{1}{q}[\bar{g}(\bar D_V \bar{J}U+p\pi(U)\bar{J}V+q\pi(U)V-\pi(\bar{J}U)\bar{J}V,\bar{J}Z)]+\frac{p}{q}\bar{g}(\bar D_V U,\bar{J}Z)$\\
		
		Using equations (\ref{q1}), (\ref{th4}) and (\ref{q6})\\
		
		$g([U,V],Z)=\frac{1}{q}\bar{g}(\bar h^{s}(U,\bar{J}V),\bar{J}Z)-\frac{p}{q}\bar{g}(\pi(V)w_s U,\bar{J}Z)-\frac{1}{q}\bar{g}(\bar h^{s}(V,\bar{J}U),\bar{J}Z)+$\\
		
		$\quad \frac{p}{q}\bar{g}(\pi(U)w_s V,\bar{J}Z)=0$\\
		
		Therefore, the result follows from the hypothesis. 
		
	\end{proof}
\end{theorem}
\begin{definition}
	Let $(\bar{M},\bar{J},\bar{g})$ be a metallic semi-Riemannian manifold and $\bar{\nabla }$ be the Levi-Civita connection on $\bar{M}$ with respect to $\bar{g}$. Then the distribution $D$ is parallel with respect to $\bar{\nabla}$ if $\bar{\nabla}_U V=0$ for all $U,V\in\Gamma(D)$.
\end{definition}
\begin{remark}
	In case of quarter symmetric non-metric connection, the distribution $D$ is parallel with respect to $\bar{D}$ if $\bar{D}_U V=0$ for all $U,V\in\Gamma(D)$. 
\end{remark}
\begin{theorem}
	For a screen semi-invariant lightlike submanifold $M$ of metallic semi-Riemannian manifold $\bar{M}$ with a quarter symmetric non-metric connection $\bar{D}$, the screen distribution is parallel  iff\\
	
	$-\bar{A}_{\bar{J}Z}U+ p\pi(Z){\bar{J}U}=\pi(\bar{J}Z){\bar{J}U}+ p\bar h^{*}(U,Z)$ for any $U\in\Gamma (S(TM)),Z\in\Gamma(B^{\perp})$\\
	
	\begin{proof}
		$S(TM)$ is parallel with respect to $\bar D$ iff
		\begin{equation}\label{th3}
			\bar g(\bar D_U Z,N)=0 
		\end{equation}
		\quad  for any $U\in\Gamma (S(TM)),Z\in\Gamma(B^{\perp}), N\in\Gamma(ltr(TM))$\\
		
		$\bar g(\bar D_U Z,N)=\frac{1}{q}[\bar g(\bar D_U \bar{J}Z+p\pi(Z)\bar{J}U-\pi(\bar{J}Z) \bar{J}U-p\bar D_U Z,\bar{J}N)]$\\
		
		\qquad$=\frac{1}{q}[\bar{g}(-\bar A_{\bar{J}Z}U,\bar{J}N)+\bar{g}(\bar \nabla_U^{s}\bar{J}Z,\bar{J}N)+\bar{g}(D^{l}(U,\bar{J}Z),\bar{J}N)-p\bar{g}(D_U Z,\bar{J}N)-$\\
		
		\qquad $p\bar{g}(\bar h^{l}(U,Z),\bar{J}N)-p\bar{g}(\bar h^{s}(U,Z),\bar{J}N)]+p\pi(Z)\bar{g}(\bar{J}U,\bar{J}N)-\pi(\bar{J}Z)\bar{g}(\bar{J}U,\bar{J}N)$\\
		
		\qquad$=\frac{1}{q}[\bar{g}(-\bar A_{\bar{J}Z}U,\bar{J}N)-p\bar{g}(D_U^{*} Z+\bar h^{*}(U,Z),\bar{J}N)+p\pi(Z)\bar{g}(\bar{J}U,\bar{J}N)-\pi(\bar{J}Z)\bar{g}(\bar{J}U,\bar{J}N)] $\\
		
		Hence the result follows from (\ref{c}), (\ref{a2}), (\ref{q1}), (\ref{q3}) and (\ref{th3}).
	\end{proof}
\end{theorem}
\begin{proposition}\label{p1}
	Let $(M,g,S(TM),S(TM^{\perp}))$ be a screen semi-invariant submanifold of metallic semi-Riemannian manifold $(\bar{M},\bar{g},\bar{J})$. Then the following assertion holds:\\
	The distrbution $B^{'}$ is parallel  with respect to induced connection $\nabla$ iff $ h^{s}(U,\bar{J}V)=0$, $U,V\in\Gamma(B^{'})$
\end{proposition}
\begin{proof}
	From the Gauss-Weingarten formulae,
	\begin{equation*}
		g(\nabla_U \bar{J}V,Z)= g(h^{s}(U,V),\bar{J}Z)
	\end{equation*}
	for $U,V\in\Gamma(B^{'})$ and $Z\in\Gamma(S(TM^{\perp}))$\\
	
	Therefore, the assertion follows using hypothesis .
\end{proof}

\begin{theorem}
	Let $(M,g,S(TM),S(TM^{\perp}))$ be a screen semi-invariant submanifold of metallic semi-Riemannian manifold $(\bar{M},\bar{g},\bar{J})$. The distribution $B^{'}$ is parallel with respect to quarter symmetric non-metric connection $D$ if and only if $B^{'}$ is parallel with respect to $\nabla$.
\end{theorem}
\begin{proof}
	For any $U,V\in\Gamma(B^{'})$,\quad$Z\in\Gamma(B^{\perp})$ \quad  $w_s U=0$, \\
	
	Therefore,
	\begin{equation*}
		\bar{h^{s}}(U,\bar{J}V)=h^{s}(U,\bar{J}V)
	\end{equation*}
	Thus the result follows from proposition $(\ref{p1}).$\\
\end{proof}
From proposition  $(\ref{p1})$, we have the following corollary.
\begin{Corollary}
	Let $(M,g,S(TM),S(TM^{\perp}))$ be a screen semi-invariant submanifold of metallic semi-Riemannian manifold $\bar{M}$. Then the following assertions are equivalent:\\ 
	(a) The distribution $B^{'}$ is parallel with respect to quarter symmetric non-metric connection $D$\\
	(b) $ \bar h^{s}(U,\bar{J}V)=0$, $\forall$ $U,V\in\Gamma(B^{'})$\\
	(c) $B^{'}$ is parallel  with respect to $\nabla$\\
	(d) $ h^{s}(U,\bar{J}V)=0$, $\forall$  $U,V\in\Gamma(B^{'})$.\\
\end{Corollary}
\begin{remark}
	Since $\bar{h}$ is not symmetric, therefore $\bar{h}(Y,X)$ may not be zero, if $\bar{h}(X,Y)=0$. 
\end{remark}
\begin{definition}
	Let $(M,g,S(TM),S(TM^{\perp}))$ be a screen semi-invariant submanifold of metallic semi-Riemannian manifold $(\bar{M},\bar{g},\bar{J})$ with a quarter symmetric non-metric connection $\bar{D}$. If $\bar{D}_U V \in\Gamma(B)$ for any $U,V\in\Gamma(B)$, then $B$ defines totally geodesic foliation in $M$.
\end{definition}

\begin{theorem}
	Suppose $M$ be a screen semi-invariant submanifold of a metallic semi-Riemannian manifold $\bar{M}$ with a quarter symmetric non metric connection. Then the distribution $B$ defines a totally geodesic foliation in $S(TM)$ iff \\
	\[\bar J \bar h^{s}(U,V)=p\pi(V)\bar{J}U+q\pi(V)U-\pi(\bar J V)\bar J U\]\\
	for all $U\in\Gamma(TM),V\in\Gamma(B),Z\in\Gamma(B^{\perp})$.\\
	\begin{proof}
		$\bar{g}(\bar D_U V,\bar{J}Z)=\bar{g}(\bar D_U \bar{J}V,Z)+p\pi(V)\bar{g}(\bar{J}U,Z)+q\pi(V)\bar{g}(U,Z)-\pi(\bar{J}V)\bar{g}(\bar{J}U,Z)$\\
		
		\qquad$=\bar{g}(D_U ^{*} \bar{J}V,Z)+\bar{g}(\bar h^{*}(U,\bar{J}V),Z)+p\pi(V)\bar{g}(\bar{J}U,Z)+q\pi(V)\bar{g}(U,Z)-\pi(\bar{J}V)\bar{g}(\bar{J}U,Z)$\\
		
		$B$ defines totally geodesic foliation in $S(TM)$ iff $g(D_U ^{*}\bar{J}V,Z)=0$\\
		
		$\bar{g}(\bar D_U V,\bar{J}Z)=\bar{g}(D_U V,\bar{J}Z)+\bar{g}(\bar h^{l}(U,V),\bar{J}Z)+\bar{g}(\bar h^{s}(U,V),\bar{J}Z)$\\
		
		\qquad$=\bar{g}(\bar{J}\bar h^{s}(U,V),Z)$\\
		
		Therefore, $\bar{g}(\bar{J}\bar h^{s}(U,V),Z)= p\pi(V)\bar{g}(\bar{J}U,Z)+q\pi(V)\bar{g}(U,Z)-\pi(\bar{J}V)\bar{g}(\bar{J}U,Z)$\\

	\end{proof}
\end{theorem} 

\begin{theorem}
	Let $M$ be a screen semi-invariant lightlike submanifold of $(\bar M,\bar g,\bar{J})$ with a quarter symmetric non-metric connection.Then  $B^{'}$ defines a totally geodesic foliation on $M$ iff\\
	
	$\bar{g}(\bar{J}E,\bar D^{l}(U,\bar{J}Z))=\bar{g}(\bar A_{\bar{J}Z} U,\bar{J}E)+\pi(\bar{J}Z)\bar{g}(\bar{J}U,\bar{J}E)$\\
	
	$\bar{g}(\bar A_{\bar{J}Z} U,\bar{J}F)=-\pi(\bar{J}Z)\bar{g}(\bar{J}U,\bar{J}F)$\\
	
	for any $E\in\Gamma(Rad(TM)),U\in\Gamma(B^{'}),Z\in\Gamma(B^{\perp}),F\in\Gamma(B)$.\\
	
	\begin{proof}
		$B^{'}$ defines totally geodesic foliation on $M$ if and only if
		\begin{equation}\label{th6}
			g(D_U V,Z)=0,\quad g(D_U \bar{J}V,Z)=0 \quad \forall\quad U,V \in\Gamma(B^{'})
		\end{equation}\\
		$g(D_U V,Z)=\frac{1}{q}[\bar{g}(\bar D_U \bar{J}V,\bar{J}Z)+p\pi(V)\bar{g}(\bar{J}U,\bar{J}Z)+q\pi(V)\bar{g}(U,\bar{J}Z)-\pi(\bar{J}V)\bar{g}(\bar{J}U,\bar{J}Z)-$\\
		
		\qquad$p\bar{g}(\bar D_U \bar{J}V,Z)-p^{2}\pi(V)\bar{g}(\bar{J}U,Z)-pq\pi(V)\bar{g}(U,Z)+p\pi(\bar{J}V)\bar{g}(\bar{J}U,Z)]$\\
		
		\qquad$=\frac{1}{q}[\bar{g}(D_U \bar{J}V,\bar{J}Z)+\bar{g}(\bar h^{l}(U,\bar{J}V),\bar{J}Z)+\bar{g}(\bar h^{s}(U,\bar{J}V),\bar{J}Z)-p\bar{g}(D_U \bar{J}V,Z)-$\\
		
		\qquad$p\bar{g}(\bar h^{l}(U,\bar{J}V),Z)-p\bar{g}(\bar h^{s}(U,\bar{J}V),Z)]$\\
		
		Following (\ref{c}), (\ref{a2}), (\ref{q1}) and (\ref{th6})\\
		
		\qquad$\bar{g}(\bar h^{s}(U,\bar{J}V),\bar{J}Z)=0$\\
		
		Using (\ref{S3}) and (\ref{t1}), $w_s U=0\quad w_l U=0$ \\
		
		$\bar{g}(\bar{J}V,\bar D^{l}(U,\bar{J}Z))=\bar{g}(\bar A_{\bar{J}Z}U,\bar{J}V)+\pi(\bar{J}Z)\bar{g}(\bar{J}U,\bar{J}V)$\\
		
		Case-1 Take $V=E\in\Gamma(Rad(TM))$ \\  
		
		$\bar{g}(\bar{J}E,\bar D^{l}(U,\bar{J}Z))=\bar{g}(\bar A_{\bar{J}Z}U,\bar{J}E)+\pi(\bar{J}Z)\bar{g}(\bar{J}U,\bar{J}E)$\\ 
		
		Case-2 Take $V=F\in\Gamma(B)$ ,$\bar{g}(\bar{J}F,\bar D^{l}(U,\bar{J}Z))=0$\\
		
		$\bar{g}(\bar A_{\bar{J}Z} U,\bar{J}F)=-\pi(\bar{J}Z)\bar{g}(\bar{J}U,\bar{J}F)$  
		
	\end{proof}
	
\end{theorem}

\begin{theorem}
	Let $M$ be a screen semi-invariant lightlike submanifold of $(\bar M,\bar g,\bar{J})$ with quarter symmetric non-metric connection $\bar{D}$. Then $B^{\perp}$ defines totally geodesic foliation on $M$ iff\\
	
	$\bar{g}(\bar h^{s}(U,\bar{J}Z),\bar{J}V)=\pi(\bar{J}Z)\bar{g}(\bar{J}U,\bar{J}V)$\\
	
	$\bar{g}(\bar D^{s}(U,\bar{J}N),\bar{J}V)=\pi(\bar{J}N)\bar{g}(\bar{J}U,\bar{J}V)$\\
	
	for any $U,V\in\Gamma(B^{\perp}),Z\in\Gamma(B),N\in\Gamma(ltr(TM))$.\\
	
	\begin{proof}
		
		$g(D_U V,\bar{J}Z)=\bar{g}(\bar D_U \bar{J}V,Z)+p\pi(V)\bar{g}(\bar{J}U,Z)+q\pi(V)\bar{g}(U,Z)-\pi(\bar{J}V)\bar{g}(\bar{J}U,Z)$\\
		
		\qquad$=\bar{g}(-\bar A_{\bar{J}V}U,Z)+\bar{g}(\bar{\nabla}_U ^{s} \bar{J}V,Z)+\bar{g}(\bar D^{l}(U,\bar{J}V),Z)$\\
		
		Following (\ref{c}), (\ref{a2}), (\ref{q3}) and (\ref{S1})\\ 
		
		\qquad$\bar{g}(-\bar A_{\bar{J}V}U,Z)=0$\\
		
		$g(D_U V,\bar{J}N)=\bar{g}(\bar D_U \bar{J}V,N)+p\pi(V)\bar{g}(\bar{J}U,N)+q\pi(V)\bar{g}(U,N)-\pi(\bar{J}V)\bar{g}(\bar{J}U,N)$\\
		
		\qquad$=\bar{g}(-\bar A_{\bar{J}V}U+\bar{\nabla}_U ^{s} \bar{J}V+\bar D^{l}(U,\bar{J}V),N)=\bar{g}(-\bar A_{\bar{J}V}U,N)=0$\\ 
		
		Using (\ref{q15}) and (\ref{q16})\\
		
		$\bar{g}(\bar h^{s}(U,\bar{J}Z),\bar{J}V)=\bar{g}(\bar A_{\bar{J}V}U,\bar{J}Z)+\pi(\bar{J}Z)\bar{g}(\bar{J}U,\bar{J}V)$\\
		
		$\bar{g} (\bar D^{s}(U,\bar{J}N),\bar{J}V)=\bar{g}(\bar A_{\bar{J}V}U,\bar{J}N)+\pi(\bar{J}N)\bar{g}(\bar{J}U,\bar{J}V)$\\
		
	\end{proof}
\end{theorem}

\bibliographystyle{amsplain}

\end{document}